%% file: manuscript.tex
\documentclass[11pt]{article}
\usepackage[margin=1in]{geometry}
\usepackage{amsmath,amssymb,amsthm,mathtools}
\usepackage{booktabs}
\usepackage{enumitem}
\usepackage[hidelinks]{hyperref}

\hypersetup{
  pdftitle={A Sharp Four-Layer Theorem for Integer-Occupied Slices of a Planar Disk-Slab},
  pdfauthor={Xinyue Liu},
  pdfsubject={Planar lattice geometry and integer programming},
  pdfkeywords={lattice enumeration, arithmetic width, Babai nearest plane, planar convex bodies, Bernstein certificates}
}

\newtheorem{theorem}{Theorem}[section]
\newtheorem{lemma}[theorem]{Lemma}
\newtheorem{proposition}[theorem]{Proposition}
\newtheorem{corollary}[theorem]{Corollary}
\theoremstyle{definition}
\newtheorem{definition}[theorem]{Definition}
\theoremstyle{remark}
\newtheorem{remark}[theorem]{Remark}

\newcommand{\Z}{\mathbb Z}

\newcommand{\dist}{\operatorname{dist}}

\title{A Sharp Four-Layer Theorem for Integer-Occupied Slices of a
Planar Disk--Slab}
\author{Xinyue Liu\thanks{Corresponding author: \texttt{liuxinyue743@gmail.com}}\\
\small Institute of Nuclear and New Energy Technology, Tsinghua University,\\
\small Beijing 100084, China}
\date{August 2026}

\begin{document}
\maketitle

\begin{abstract}
Let a rational positive-definite quadratic sublevel set on a full-rank affine
lattice coset be intersected with a closed slab.  A deterministic
two-dimensional $\delta=3/4$ reduction selects a primitive-dual integer
functional and an inner ellipse.  In the branch where the corresponding
Babai point lies outside that ellipse, the lattice points in the disk--slab
occupy at most four integer levels of the selected functional.  An explicit
rational instance attains four, so the bound is sharp.  A strict diameter
estimate reduces any counterexample to three blocks of five consecutive
levels, and a nearest-integer inequality excludes the central block.
Symmetry and cell localization reduce the two side blocks to eight modes:
four follow from a shared-cap inequality, and four from an exact
four-variable Bernstein certificate.  The certificate contains $17{,}640$
positive coefficients, with minimum $625/2048$.
\end{abstract}



\section{Introduction}

For a convex body $K$ and an integer functional $w$, directional arithmetic
width counts the values attained by $w$ on $K\cap\mathbb Z^d$
\cite{DeLoeraMarstersONeill2025}.  Deza and Onn introduced a related
vertex-based layer number \cite{DezaOnn1995}.  These notions differ from
classical lattice width, which minimizes the length of a real projection over
primitive directions.  Planar flatness results for hollow bodies and for
bodies with prescribed interior lattice points belong to that latter theory
\cite{Hurkens1990,AverkovCodenottiFreyerHuang2026,CodenottiHallHofscheier2021,KannanLovasz1988}.

Here the direction is not optimized.  It is selected by a deterministic
reduction of the quadratic form.  Moreover, failure of the associated Babai
point does not imply that the disk--slab is lattice-free; many feasible
lattice points may remain.  Nearest-plane algorithms, lattice enumeration,
and fixed-dimensional integer programming provide the algorithmic setting
\cite{Babai1986,Lenstra1983,FinckePohst1985,SchnorrEuchner1994,AgrellEtAl2002}.
Those general methods do not determine how many levels of this selected
functional can be occupied.

For the strict-failure branch defined below, the main result is $L_Z^*=4$.
A symmetric estimate first gives seven safe candidate levels, and a strict
diameter bound limits the real body to five consecutive levels.  The theorem
excludes simultaneous lattice occupation of all five.  A rational example
occupies four levels while the real body intersects five, showing that the
integer and continuous counts are distinct.

The proof treats the three possible five-level blocks separately.  A
nearest-integer inequality excludes the central block.  Layer reversal
identifies the two side blocks, and cell localization reduces the remaining
case to eight modes.  Four modes violate a shared-cap inequality directly;
the other four are excluded by an exact Bernstein-coefficient certificate.

\section{Setting and notation}
\label{sec:contract}

Let
\begin{equation}
 F(\mathbf q)=\mathbf q^TM\mathbf q+2r_0^T\mathbf q+c_0,
 \qquad M=M^T\in\mathbb Q^{2\times2},\quad M\succ0,
 \label{eq:original-quadratic}
\end{equation}
and let the full-rank affine lattice coset be
\begin{equation}
 \mathbf q=q_0+Kz,\qquad q_0\in\mathbb Z^2,\quad
 K\in\mathbb Z^{2\times2},\quad \det K\ne0,\quad z\in\mathbb Z^2.
 \label{eq:affine-coset}
\end{equation}
Let $u_1=(1,0)^T$ and $u_2=(0,1)^T$ be the standard coordinate vectors.
For integers $A_0\le B_0$, the original
real disk--slab is
\begin{equation}
 \mathcal C^{\mathbb R}_{A_0,B_0}=
 \{\mathbf q\in\mathbb R^2:F(\mathbf q)\le0,
 A_0\le u_1^T\mathbf q\le B_0\}.
 \label{eq:original-body}
\end{equation}
All inequalities in this paper are closed.

Complete the square by setting
\begin{equation}
 q_c=-M^{-1}r_0,\qquad R=r_0^TM^{-1}r_0-c_0,\qquad Q=M/R.
 \label{eq:completion}
\end{equation}
Thus $F(\mathbf q)=(\mathbf q-q_c)^TM(\mathbf q-q_c)-R$.  The discussion
below concerns the nondegenerate branch $R>0$.  Put
\[
 \sigma^2=u_1^TQ^{-1}u_1,\qquad
 c_1=u_1^Tq_c.
\]
The real projection is $[c_1-\sigma,c_1+\sigma]$.  Before constructing an
ellipse, close the slab exactly on integer coordinates:
\begin{equation}
 A=\max\{A_0,\lceil c_1-\sigma\rceil\},\qquad
 B=\min\{B_0,\lfloor c_1+\sigma\rfloor\}.
 \label{eq:integer-clamp}
\end{equation}
The theorem concerns the rank-two case $A<B$; empty, singleton, and
fixed-column cases are separate one-dimensional branches.

Set
\begin{equation}
 p_s=\frac{A+B}{2}-c_1,\qquad
 \sigma^2h_s^2=\frac{(B-A)^2}{4},\qquad
 s_s^2=\frac{\sigma^2-p_s^2}{4\sigma^2}.
 \label{eq:rounding-scalars}
\end{equation}
The rational centre and metric
\begin{align}
 q_E&=q_c+\frac{p_s}{\sigma^2}Q^{-1}u_1,\label{eq:qE}\\
 H_E&=\frac{Q}{s_s^2}+
 \left(\frac1{\sigma^2h_s^2}-\frac1{\sigma^2s_s^2}\right)
 u_1u_1^T
 \label{eq:HE}
\end{align}
define the inner ellipse.  To see the geometry directly, use disk coordinates
\[
 x=\frac{u_1^T(\mathbf q-q_c)}{\sigma},\qquad
 y^2=(\mathbf q-q_c)^T
 \left(Q-\frac{u_1u_1^T}{\sigma^2}\right)
 (\mathbf q-q_c).
\]
Then the full quadratic body is $x^2+y^2\le1$ and the clamped slab is
$a\le x\le b$, where $a=(A-c_1)/\sigma$ and $b=(B-c_1)/\sigma$.
With $m_s=(a+b)/2$, $h_s=(b-a)/2$, and
$X=(x-m_s)/h_s$, $Y=y/s_s$, the ellipse
$X^2+Y^2\le1$ is exactly
$(\mathbf q-q_E)^TH_E(\mathbf q-q_E)\le1$.  The endpoint conditions imply
\[
 1-(m_s+h_sX)^2\ge\frac{(1-m_s^2)(1-X^2)}4
\]
and $h_s\le1-|m_s|$.  Conversely,
$(1-x^2)/(1-m_s^2)\le2$ on the slab.  These two inequalities give the
factor-three rounding
\begin{equation}
 \{\mathbf q:(\mathbf q-q_E)^TH_E(\mathbf q-q_E)\le1\}
 \subseteq\mathcal C^{\mathbb R}_{A,B}
 \subseteq
 \{\mathbf q:(\mathbf q-q_E)^TH_E(\mathbf q-q_E)\le9\}.
 \label{eq:three-rounding}
\end{equation}

Pull the metric back through \eqref{eq:affine-coset}:
\begin{equation}
 z_E=K^{-1}(q_E-q_0),\qquad H_z=K^TH_EK.
 \label{eq:pullback}
\end{equation}
Let $D$ be the least positive common multiple of the canonical denominators
of the three independent entries of $H_z$, and put $G=DH_z$.  Exact
deterministic $\delta=3/4$ reduction returns $U\in SL(2,\mathbb Z)$ with
\begin{equation}
 U^TGU=\begin{pmatrix}a_G&b_G\\b_G&c_G\end{pmatrix},\qquad
 \mu=\frac{b_G}{a_G},\qquad
 \eta=c_G-\frac{b_G^2}{a_G},
 \label{eq:reduced-gram}
\end{equation}
where $-1/2\le\mu<1/2$, $c_G\ge3a_G/4$, and
$\eta\ge a_G/2$.  The half-open endpoint is fixed by using nearest-integer
rounding with every half tie sent to the larger integer.

Write $\xi=U^{-1}z_E$ and define
\begin{align*}
 N_2&=\operatorname{nint}(\xi_2),& e_2&=\xi_2-N_2,\\
 N_1&=\operatorname{nint}(\xi_1+\mu e_2),&
 e_1&=\xi_1+\mu e_2-N_1.
\end{align*}
The deterministic Babai point is $p_B=U(N_1,N_2)^T$.  Centre the reduced
lattice coordinates by
\begin{equation}
 \binom tj=U^{-1}z-\binom{N_1}{N_2},\qquad
 \alpha=\frac{a_G}{D},\qquad \beta=\frac\eta D.
 \label{eq:centered-reduced}
\end{equation}
In these coordinates the inner-ellipse gauge is
\begin{equation}
 \mathcal H(t,j)=\alpha(t-e_1+\mu j)^2+\beta(j-e_2)^2,
 \label{eq:gauge}
\end{equation}
where
\begin{equation}
 \alpha>0,\qquad -\tfrac12\le\mu<\tfrac12,\qquad
 \beta\ge\alpha(\tfrac34-\mu^2),\qquad
 e_1,e_2\in[-\tfrac12,\tfrac12).
 \label{eq:reduced-contract}
\end{equation}
The Babai point is now $(0,0)$.  The theorem concerns the strict complement
of the closed inner ellipse,
\begin{equation}
 E_B:=\alpha e_1^2+\beta e_2^2>1.
 \label{eq:babai-fail}
\end{equation}
This condition is exactly
$(p_B-z_E)^TG(p_B-z_E)>D$.  Since
$E_B\le(\alpha+\beta)/4$ and $\beta\ge\alpha/2$, it implies
\begin{equation}
 \beta>\tfrac43
 \label{eq:beta-floor}
\end{equation}
and \eqref{eq:three-rounding} gives, for every occupied lattice point,
\begin{equation}
 \mathcal H(t,j)\le9.
 \label{eq:outer3}
\end{equation}

The reduction also identifies the counted functional.  Let
\begin{equation}
 w=U^{-T}u_2.
 \label{eq:primitive-dual}
\end{equation}
It is a primitive integer vector, and the centred level of $z$ is
$j=w^Tz-N_2$.  Equivalently, on the original affine coset,
\[
 j=w^TK^{-1}(\mathbf q-q_0)-N_2.
\]

Finally, the original body has structure that cannot be replaced by the outer
ellipse alone.  Reflect the normalized first axis when necessary and put
$m=|m_s|$, $h=h_s$, so that $x=m+hX$.  Since
$s_s^2=(1-m^2)/4$, the disk inequality $x^2+y^2\le1$ becomes
\[
 Y^2\le4\frac{1-(m+hX)^2}{1-m^2}.
\]
Define
\[
 \lambda=\frac{h}{1-m},\qquad \rho=\frac{h}{1+m}.
\]
The right-hand side factors without approximation:
\begin{align*}
4\frac{1-(m+hX)^2}{1-m^2}
&=4\frac{(1-m-hX)(1+m+hX)}{(1-m)(1+m)}\\
&=4(1-\lambda X)(1+\rho X)\\
&=4\{1-(\lambda-\rho)X-\lambda\rho X^2\}.
\end{align*}
The slab endpoint inequalities give $0<\rho\le\lambda\le1$.
Thus every point of one instance is measured in the same normalized axes and
with the same pair $(\lambda,\rho)$.
Let
\[
 \mathcal C_{\rm red}=\left\{(t,j)\in\mathbb R^2:
 q_0+KU\binom{N_1+t}{N_2+j}\in\mathcal C^{\mathbb R}_{A,B}\right\},
\]
and let $\widetilde{\mathcal C}$ be its image in the common inner-ellipse
coordinates $(X,Y)$.  Every point of $\widetilde{\mathcal C}$ obeys
\begin{equation}
|X|\le1,\qquad
 Y^2\le4\{1-(\lambda-\rho)X-\lambda\rho X^2\},
 \qquad0<\rho\le\lambda\le1.
 \label{eq:common-cap}
\end{equation}
The axes and the pair $(\lambda,\rho)$ are shared by all points of one
instance.  More precisely, for a displacement $(\Delta t,\Delta j)$ define
\[
 \|(\Delta t,\Delta j)\|_H^2
 =\alpha(\Delta t+\mu\Delta j)^2+\beta(\Delta j)^2.
\]
The affine map to $(X,Y)$ has orthogonal linear part with respect to this
inner-ellipse metric, so for any $v_1,v_2\in\mathcal C_{\rm red}$ and their
images $(X_1,Y_1),(X_2,Y_2)$,
\begin{equation}
 \|v_1-v_2\|_H^2=(X_1-X_2)^2+(Y_1-Y_2)^2.
 \label{eq:isometry}
\end{equation}

Let
\[
 I=(M,r_0,c_0,q_0,K,A_0,B_0)
\]
denote an admissible instance together with the deterministic reduction and
tie-breaking conventions fixed above.  The supremum below ranges over all
such instances in the nondegenerate strict-failure branch.

\begin{definition}
For data \eqref{eq:original-quadratic}--\eqref{eq:original-body} satisfying
the nondegenerate strict-failure assumptions above, define the occupied
integer levels of the reduction-induced primitive-dual affine functional by
\[
 J_Z(I)=\left\{j\in\Z:\begin{array}{l}
 \exists t\in\Z\text{ such that }\mathbf q=
 q_0+KU\binom{N_1+t}{N_2+j},\\[-2pt]
 F(\mathbf q)\le0,\quad A\le u_1^T\mathbf q\le B
 \end{array}\right\}.
\]
Set $L_Z^*=\sup_I|J_Z(I)|$.
\end{definition}

The set $J_Z(I)$ records only levels containing lattice points.  The
seven-element safe list and the levels intersecting the real body are
auxiliary supersets and are counted separately.

\section{From five occupied levels to three configurations}
\label{sec:three-configurations}

The following continuous estimate is included because it is used twice in
the discrete argument.

\begin{lemma}[Strict disk--slab diameter]
\label{lem:disk-slab-diameter}
In the inner-ellipse coordinates of \eqref{eq:common-cap}, the Euclidean
diameter of $\widetilde{\mathcal C}$ is strictly smaller than $4\sqrt2$.
\end{lemma}

\begin{proof}
Write
\[
 g(X)^2=1-(\lambda-\rho)X-\lambda\rho X^2.
\]
For two points, set $X_1=u+v$ and $X_2=u-v$; then
$|u|+|v|\le1$.  Their squared distance is at most
\[
 4v^2+4\{g(X_1)+g(X_2)\}^2.
\]
By Cauchy--Schwarz and direct expansion,
\[
 \{g(X_1)+g(X_2)\}^2
 \le4-4(\lambda-\rho)u-4\lambda\rho(u^2+v^2).
\]
If $u\ge0$, the right-hand side is at most $4$, and the squared distance is
at most $20<32$.  If $u<0$, then $\lambda-\rho<1$ because
$0<\rho\le\lambda\le1$; hence the displayed bound is strictly smaller than
$4+4(-u)\le8-4|v|\le8-v^2$.  The squared distance is therefore strictly
smaller than $4v^2+4(8-v^2)=32$ for every pair of points.  The original
closed quadratic--slab body is bounded, hence compact; its image
$\widetilde{\mathcal C}$ is
compact as well.  The continuous squared-distance function on
$\widetilde{\mathcal C}\times\widetilde{\mathcal C}$
therefore attains its maximum at a pair to which the strict estimate applies.
Consequently,
\[
 \max_{P_1,P_2\in\widetilde{\mathcal C}}
 \|P_1-P_2\|_2^2<32,
\]
and therefore
$\operatorname{diam}(\widetilde{\mathcal C})<4\sqrt2$.
\end{proof}

\begin{proposition}[Five-level trichotomy]
\label{prop:five-level-trichotomy}
Every admissible instance satisfies $|J_Z(I)|\le5$.  If
$|J_Z(I)|\ge5$, then equality holds and
\begin{equation}
J_Z(I)\in\bigl\{
 \{-3,-2,-1,0,1\},
 \{-2,-1,0,1,2\},
 \{-1,0,1,2,3\}
\bigr\}.
\label{eq:three-configs}
\end{equation}
\end{proposition}

\begin{proof}
For an occupied point $(t,j)$, \eqref{eq:outer3} gives
\[
 \beta(j-e_2)^2\le9,
 \qquad
 |j-e_2|\le\frac3{\sqrt\beta}
 <\frac{3\sqrt3}{2}<3.
\]
Since $e_2\in[-1/2,1/2)$, every occupied level lies in
$\{-3,-2,-1,0,1,2,3\}$; in particular, $J_Z(I)$ is finite.

If $J_Z(I)$ is nonempty, let $j_-$ and $j_+$ be its smallest and largest
elements and choose occupied points $v_-=(t_-,j_-)$ and
$v_+=(t_+,j_+)$.  By \eqref{eq:isometry},
Lemma~\ref{lem:disk-slab-diameter}, and $\beta>4/3$,
\[
 \sqrt\beta\,(j_+-j_-)
 \le\|v_+-v_-\|_H
 <4\sqrt2.
\]
Hence
\[
 j_+-j_-<\frac{4\sqrt2}{\sqrt\beta}<2\sqrt6<5.
\]
The left side is an integer, so $j_+-j_-\le4$ and
\[
 |J_Z(I)|\le j_+-j_-+1\le5.
\]
The same conclusion is trivial when $J_Z(I)$ is empty.

If at least five levels are occupied, the preceding inequalities force
$|J_Z(I)|=5$ and $j_+-j_-=4$.  All five integers between the two endpoints
must then be occupied.  A block of five consecutive integers inside the
seven-element safe list can start only at $-3$, $-2$, or $-1$, which gives
exactly the three blocks in \eqref{eq:three-configs}.
\end{proof}

Layer reversal exchanges the two side configurations.  It remains to exclude
the central configuration and one normalized side configuration.

\section{Central five-level exclusion}
\label{sec:central-exclusion}

\begin{lemma}[Nearest-integer inequality]
\label{lem:nearest-integer}
For $x\in[-1/2,1/2]$ and $y\in[-1,1]$,
\begin{equation}
 \dist(x-y,\Z)^2+3\ge9x^2+y^2.
 \label{eq:nearest-lemma}
\end{equation}
The inequality includes all nearest-integer tie points.
\end{lemma}

\begin{proof}
A nearest integer can be chosen from $\{-1,0,1\}$.  In the zero cell, put
$s=x-y$ with $|s|\le1/2$.  The left side minus the right side is
$3-8x^2-2xy$.  For $x\ge0$ this is minimized at $y=x+1/2$, where it is
$3-10x^2-x\ge0$; the $x\le0$ case follows by simultaneous sign reversal.
In the $+1$ cell, $s=x-y\in[1/2,3/2]$ and the difference is
$4-10x^2+2s(x-1)$.  Since $x-1<0$ and $s\le x+1$, it is at least
$2-8x^2\ge0$.  The $-1$ cell is symmetric.  A tie changes only the chosen
representative, not the distance.
\end{proof}

\begin{proposition}[Central block exclusion]
\label{prop:central-block-exclusion}
The block $\{-2,-1,0,1,2\}$ cannot be contained in $J_Z(I)$.
\end{proposition}

\begin{proof}
Put $\vartheta=|e_2|$ and choose $j_f\in\{-2,2\}$ so that
$|j_f-e_2|=2+\vartheta$.  Set
\[
 d=\dist(e_1-\mu j_f,\Z).
\]
Occupancy of level $j_f$, \eqref{eq:babai-fail}, and
\eqref{eq:outer3} give
\[
 \alpha d^2+\beta(2+\vartheta)^2\le9
 <9\alpha e_1^2+9\beta\vartheta^2.
\]
Since $0\le\vartheta\le1/2$,
$(2+\vartheta)^2-9\vartheta^2=4+4\vartheta-8\vartheta^2\ge4$.
Consequently,
\begin{align*}
0
&>\alpha(d^2-9e_1^2)
 +\beta\{(2+\vartheta)^2-9\vartheta^2\}\\
&\ge\alpha(d^2-9e_1^2)+4\beta\\
&\ge\alpha\{d^2-9e_1^2+3-4\mu^2\},
\end{align*}
where the last line uses $\beta/\alpha\ge3/4-\mu^2$.
On the other hand, $|\mu j_f|\le1$, so
Lemma~\ref{lem:nearest-integer}, applied with $x=e_1$ and
$y=\mu j_f$, yields
\[
 d^2+3\ge9e_1^2+(\mu j_f)^2=9e_1^2+4\mu^2.
\]
This is the reverse of the preceding strict inequality.
\end{proof}

\section{Side normalization and eight modes}
\label{sec:eight-modes}

Normalize the side configuration to $\{-1,0,1,2,3\}$.  Layer reversal is
the coordinate change
\[
 (t,j,e_2,\mu)\longmapsto(t,-j,-e_2,-\mu).
\]
It preserves both the gauge and occupation.  If either $e_2=-1/2$ or
$\mu=-1/2$, the corresponding reflected
value $1/2$ is retained only in the larger closed relaxation
$e_2,\mu\in[-1/2,1/2]$ used below.  Thus the normalization does not discard a
boundary configuration, while the original deterministic tie convention
remains half-open.

\begin{proposition}[Side localization and complete mode list]
\label{prop:side-eight-modes}
Assume that the normalized side block $\{-1,0,1,2,3\}$ is occupied.
After the allowed first-coordinate reflection, write
\[
 \vartheta=e_2>0,\qquad e=|e_1|,\qquad h=|\mu|.
\]
Then
\begin{equation}
 e>\frac{47}{100},\qquad h>\frac{11}{25},\qquad
 \vartheta>\frac{23}{50},
 \label{eq:ehq}
\end{equation}
and
\begin{equation}
 \frac{64}{25}<\alpha\le\frac{72}{25},\qquad
 \frac43<\beta\le\frac{36}{25},\qquad
 \frac12\le r:=\frac\beta\alpha<\frac9{16}.
 \label{eq:alphabeta}
\end{equation}
Moreover, the occupied points on rows $3$ and $2$ must belong to exactly one
of the eight modes in \eqref{eq:eight-mode-table} below.
\end{proposition}

\begin{proof}
If the transformed $e_2\le0$, every point on row $3$ satisfies
\[
 \mathcal H(t,3)\ge\beta(3-e_2)^2\ge9\beta>9,
\]
contrary to occupancy.  Hence $\vartheta=e_2>0$.

Replacing the first reduced basis vector by its negative sends
$(t,e_1,\mu)$ to $(-t,-e_1,-\mu)$ and leaves
$t-e_1+\mu j$ unchanged up to sign.  Hence $e=|e_1|$ may be used
and retain the two signs of $\mu$ as separate families.  Any newly included
half-open endpoint is used only in a larger closed relaxation.

Let $d=\dist(e-3\mu,\Z)$.  Row-$3$ occupancy and strict Babai failure give
\[
 \alpha d^2+\beta(3-\vartheta)^2
 <9\alpha e^2+9\beta\vartheta^2.
\]
Since
\[
 (3-\vartheta)^2-9\vartheta^2
 =9-6\vartheta-8\vartheta^2\ge4,
\]
and $\beta/\alpha\ge3/4-h^2$, it follows that
\begin{equation}
 9e^2+4h^2-d^2>3.
 \label{eq:side-base}
\end{equation}
If $e\le2/5$, the left side of \eqref{eq:side-base} is at most
\[
 9(2/5)^2+4(1/2)^2=\frac{61}{25}<3.
\]
If $h\le2/5$, it is at most
\[
 9(1/2)^2+4(2/5)^2=\frac{289}{100}<3.
\]
Thus $e,h>2/5$.  If $e\le47/100$, the same expression is at most
\[
 9(47/100)^2+1=\frac{29881}{10000}<3.
\]
For $2/5<h\le11/25$, the nearest integer to $e-3h$ is $-1$, whereas
the nearest integer to $e+3h$ is $2$.  In either family the left side of
\eqref{eq:side-base} is maximized at $e=1/2$ and equals $9h-5h^2$.
It is therefore at most
\[
 9\frac{11}{25}-5\left(\frac{11}{25}\right)^2
 =\frac{1870}{625}<3.
\]
This proves the bounds on $e$ and $h$ in \eqref{eq:ehq}.

If $\vartheta\le23/50$, row-$3$ occupancy gives
$\beta\le9/(3-\vartheta)^2$.  Reduction gives $\alpha\le2\beta$, and hence
\[
E_B\le\beta(1/2+\vartheta^2)
\le\frac{9(1/2+\vartheta^2)}{(3-\vartheta)^2}
\le\frac{16011}{16129}<1.
\]
This contradiction proves $\vartheta>23/50$.
Row-$3$ occupancy gives $\beta\le36/25$, while strict failure and
$e^2,\vartheta^2\le1/4$ give $\alpha+\beta>4$.  Together with
$\alpha\le2\beta$, these inequalities prove \eqref{eq:alphabeta}.

For $\mu=+h$, the row-$3$ centre $e-3h$ lies in
$(-103/100,-82/100)$, so its unique nearest integer is $t_3=-1$.
For $\mu=-h$, the centre $e+3h$ lies in $(179/100,2]$, so $t_3=2$.
Indeed, the row-$3$ vertical energy is greater than $25/3$, so
\[
 |t_3-(e-3\mu)|<\sqrt{\frac{2}{3\alpha}}<\sqrt{\frac{25}{96}},
\]
which proves uniqueness in both centre intervals.

On row $2$ the vertical energy is greater than $3$; hence
\[
 |t_2-(e-2\mu)|<\sqrt{6/\alpha}<\sqrt{75/32}.
\]
For $\mu=+h$, the centre $e-2h$ lies in
$(-53/100,-19/50)$, and every integer outside $\{-2,-1,0,1\}$ is at
distance greater than $119/50$.  For $\mu=-h$, the centre $e+2h$ lies in
$(27/20,3/2]$, and every integer outside $\{0,1,2,3\}$ is at distance
greater than $47/20$.  Since $\alpha>64/25$, every excluded integer has
horizontal energy greater than $6$.  The complete cell table is therefore
\begin{equation}
\begin{array}{c|c|c}
\text{family}&t_3&t_2\\ \hline
\mu=+h&-1&-2,-1,0,1\\
\mu=-h& 2&0,1,2,3.
\end{array}
\label{eq:eight-mode-table}
\end{equation}
Every hypothetical side counterexample belongs to one of these eight modes.
\end{proof}

For a fixed mode define
\begin{equation}
\begin{aligned}
a_3&=t_3-e+3\mu,& b_3&=3-\vartheta,\\
a_2&=t_2-e+2\mu,& b_2&=2-\vartheta,\\
d&=t_3-t_2+\mu,\\
p&=a_3^2+r b_3^2,&t&=a_3d+r b_3,\\
g&=e^2+r\vartheta^2,&\delta&=a_3b_2-a_2b_3.
\end{aligned}
\label{eq:mode-quantities}
\end{equation}
Writing $P=\alpha p$ for the row-$3$ norm, $C$ for its dot product with row
$2$, and $\Delta$ for the determinant magnitude gives
\begin{equation}
P-C=\alpha t,\qquad
\Delta=\alpha\sqrt r\,|\delta|,\qquad E_B=\alpha g.
\label{eq:geometric-quantities}
\end{equation}
The determinant polynomial simplifies to
\[
 \delta=e-\mu\vartheta+t_3(2-\vartheta)-t_2(3-\vartheta).
\]
For the eight cells, the expression to be bounded is therefore
\[
\begin{array}{c|r|l}
\mu&t_2&\delta\\ \hline
+h&-2&4+e-\vartheta-h\vartheta\\
+h&-1&1+e-h\vartheta\\
+h& 0&e-2+\vartheta-h\vartheta\\
+h& 1&e-5+2\vartheta-h\vartheta\\
-h& 0&4+e-2\vartheta+h\vartheta\\
-h& 1&1+e-\vartheta+h\vartheta\\
-h& 2&e-2+h\vartheta\\
-h& 3&e-5+\vartheta+h\vartheta.
\end{array}
\]
Endpoint evaluation on the closed box containing \eqref{eq:ehq} gives
\begin{equation}
\begin{array}{c|r|c@{\qquad}c|r|c}
\mu&t_2&\delta\text{ interval}&\mu&t_2&\delta\text{ interval}\\ \hline
+h&-2&[93/25,4797/1250]&-h&0&[369/100,381/100]\\
+h&-1&[61/50,811/625]&-h&1&[119/100,127/100]\\
+h&0&[-13/10,-61/50]&-h&2&[-3319/2500,-5/4]\\
+h&1&[-96/25,-93/25]&-h&3&[-9669/2500,-15/4].
\end{array}
\label{eq:delta-table}
\end{equation}
Every mode therefore has $|\delta|\ge119/100$.  The four easy modes
$(+h,-2),(+h,1),(-h,0),(-h,3)$ have $|\delta|>7/2$; the remaining
four have $|\delta|<4/3$.

The remaining uniform estimates also follow from the same finite intervals.
In the two sign families,
\[
 a_3\in[-9/50,3/100]
 \quad\text{or}\quad
 a_3\in[0,21/100].
\]
The eight cells give
\[
 |a_2|\le\frac{33}{20},\qquad
 a_3d\ge-\frac{63}{200},\qquad
 |a_3a_2|\le\frac{693}{2000}.
\]
Since $r\ge1/2$ and $b_3=3-\vartheta>5/2$,
\[
 t=a_3d+rb_3
 \ge-\frac{63}{200}+\frac12\frac52
 =\frac{187}{200}>0.
\]
Moreover,
\[
\begin{aligned}
7p-30g
&=7a_3^2-30e^2
 +r\{7(3-\vartheta)^2-30\vartheta^2\}\\
&\ge-\frac{30}{4}
 +\frac12\left(7\frac{25}{4}-30\frac14\right)
 =\frac{85}{8}>0.
\end{aligned}
\]
For the row dot product,
\[
 \alpha|a_3a_2|
 \le\frac{72}{25}\frac{693}{2000}
 =\frac{6237}{6250}<1,
\]
whereas
\[
 \beta b_3b_2>
 \frac43\frac52\frac32=5.
\]
Thus $C=\alpha a_3a_2+\beta b_3b_2>4$.  In addition,
\[
 P-C=\alpha t>\frac{64}{25}\frac{187}{200}>2,
 \qquad
 P\ge\beta b_3^2>\frac43\frac{25}{4}=\frac{25}{3}.
\]
Finally, $\alpha(1+r)=\alpha+\beta>4$ and $r\ge1/2$ imply
\[
 \alpha\sqrt r>\frac{4\sqrt r}{1+r}\ge\frac{4\sqrt2}{3}.
\]
Together with \eqref{eq:delta-table}, this proves
$\Delta>119\sqrt2/75>219/100$.  The resulting uniform bounds are
\begin{equation}
t>0,\quad 7p-30g>0,\quad C>4,\quad P-C>2,\quad
\Delta>\frac{219}{100},\quad P>\frac{25}{3}.
\label{eq:uniform-bounds}
\end{equation}
In an easy mode, $\Delta>4$.  In a hard mode,
$\alpha\le72/25$, $\sqrt r<3/4$, and $|\delta|<4/3$, so
$\Delta<72/25<3$.  For later use, the hard cells also give
$|a_2|\le13/20$ and $b_2<77/50$, whence
\begin{equation}
 P_2=\alpha a_2^2+\beta b_2^2
 <\frac{72}{25}\left(\frac{13}{20}\right)^2
  +\frac{36}{25}\left(\frac{77}{50}\right)^2
 =\frac{144747}{31250}<\frac{116}{25}.
 \label{eq:P2-bound}
\end{equation}
\begin{proposition}[Certified hard-mode inequality]
\label{prop:certified-hard-mode}
For each of the four hard modes,
\begin{equation}
 \sqrt r\,|\delta|(7p-30g)>20tg.
 \label{eq:cert-inequality}
\end{equation}
\end{proposition}

\begin{proof}
On the actual side domain,
\[
 t>0,\qquad 7p-30g>0,\qquad r>0,\qquad g>0.
\]
Introduce the closed-box variables
\[
 r=\frac34-h^2+R,\qquad
 e=\frac12-E,\qquad
 h=\frac12-H,\qquad
 \vartheta=\frac12-Q_0,
\]
where
\[
 0\le R\le\frac1{16},\qquad
 0\le E\le\frac3{100},\qquad
 0\le H\le\frac3{50},\qquad
 0\le Q_0\le\frac1{25}.
\]
This closed rational box contains the actual parameter domain.  For each
mode define
\[
 \mathcal B=r\delta^2(7p-30g)^2-400t^2g^2.
\]
After scaling the box to $[0,1]^4$, the polynomial has tensor degree
$d=(4,6,8,6)$.  If $a_\nu$ is its power coefficient, its Bernstein
coefficient of degree $d$ is
\[
 b_k=\sum_{\nu\le k}a_\nu
 \prod_{i=0}^3
 \frac{\binom{k_i}{\nu_i}}{\binom{d_i}{\nu_i}}.
\]
Appendix~\ref{app:certificate} records the exact result for every mode.
There are $2205$ coefficients per mode, all are positive, and the smallest
coefficient among all eight modes is $625/2048$.  Because the Bernstein
basis is nonnegative and sums to one on the unit cube,
\[
 \mathcal B\ge\frac{625}{2048}>0.
\]
The signs displayed at the start of the proof hold on the actual domain, so
taking positive square roots gives \eqref{eq:cert-inequality}.
\end{proof}

\section{Shared-cap contradiction}
\label{sec:shared-cap}

\begin{proposition}[Exclusion of the normalized side block]
\label{prop:side-block-exclusion}
The five levels $\{-1,0,1,2,3\}$ cannot all be occupied.
\end{proposition}

\begin{proof}
Fix one of the eight modes and let $(X,Y)$ be the row-$3$ cap coordinates.
If $X\ge0$, then \eqref{eq:common-cap} gives $Y^2\le4$, and hence
\[
 P=X^2+Y^2\le5<\frac{25}{3},
\]
contradicting the bound $P>25/3$ in \eqref{eq:uniform-bounds}.
If $-4/5\le X<0$, write $X=-s$ with $0<s\le4/5$.  The factored cap bound
gives
\[
 Y^2\le4(1+\lambda s)(1-\rho s)\le4(1+s),
\]
and therefore
\[
 P=s^2+Y^2\le s^2+4(1+s)\le\frac{196}{25}<\frac{25}{3},
\]
the same contradiction.  Thus $X<-4/5$.  Write $X=-s$ and, after an
orthogonal reflection, $Y=y>0$.  Then
\begin{equation}
 \frac45<s\le1,\qquad
 y^2=P-s^2>\frac{22}{3}>\left(\frac{27}{10}\right)^2.
\label{eq:sy}
\end{equation}

The row-$2$ first projection has the two algebraic possibilities
\begin{equation}
x_2=\frac{-Cs\pm\Delta y}{P}.
\label{eq:x2}
\end{equation}
The minus sign is below $-1$ because
\[
Cs+\Delta y>
4\frac45+\frac{219}{100}\frac{27}{10}
=\frac{9113}{1000}>9\ge P.
\]
If $P_2$ denotes the row-$2$ norm, then
\[
(\Delta y)^2-(Cs)^2=P(\Delta^2-P_2s^2)>0;
\]
for easy modes $P_2\le9<16<\Delta^2$, while for hard modes
$P_2<116/25<(219/100)^2<\Delta^2$ by \eqref{eq:P2-bound}.
All quantities are positive, so $\Delta y>Cs$ and the plus numerator in
\eqref{eq:x2} is positive.  Moreover, solving the same dot-product and
determinant equations for the other coordinate gives
\[
 y_2=\frac{Cy+\Delta s}{P}>0.
\]
The strip-compatible row-$2$ point is therefore $(z,y_2)$ with
$0<z\le1$ and $y_2>0$.  Taking the positive orientation of the determinant
gives
\begin{equation}
\Delta=s y_2+zy.
\label{eq:delta}
\end{equation}

Apply \eqref{eq:common-cap} to $(-s,y)$ and $(z,y_2)$, multiply the two
inequalities by $z$ and $s$, respectively, and divide by $s+z$.
The terms linear in $X$ cancel, while $\lambda\rho sz>0$, so
\begin{equation}
W:=\frac{zy^2+s y_2^2}{s+z}
\le4(1-\lambda\rho sz)<4.
\label{eq:Wupper}
\end{equation}
Weighted Cauchy and \eqref{eq:delta} give
\begin{equation}
W\ge\left(\frac{\Delta}{s+z}\right)^2.
\label{eq:Wlower}
\end{equation}
It remains to prove $\Delta>2(s+z)$.  In an easy mode,
$\Delta>4$ and $s,z\le1$, so the inequality follows.

For a hard mode, \eqref{eq:x2} yields
\[
P(s+z)=(P-C)s+\Delta\sqrt{P-s^2}.
\]
Define
\[
\Phi(s)=2(P-C)s+2\Delta\sqrt{P-s^2}-P\Delta.
\]
The hard-mode bounds give
\[
 \Phi'(s)=2(P-C)-\frac{2\Delta s}{\sqrt{P-s^2}}
 >4-\frac{2\cdot3}{27/10}=4-\frac{20}{9}>0.
\]
Because $s\le1$, this implies $\Phi(s)\le\Phi(1)$.  For
$25/3<P\le9$, define
\[
G(P):=P-2\sqrt{P-1}.
\]
Both sides of
$3P/10+3>2\sqrt{P-1}$ are positive, and
\[
 \left(\frac{3P}{10}+3\right)^2-4(P-1)
 =\frac{(10-P)(130-9P)}{100}>0.
\]
Thus
\begin{equation}
 G(P)>\frac{7P}{10}-3.
 \label{eq:G-lower}
\end{equation}
Put $A_c=7p-30g>0$.  Since $P=\alpha p$ and $\alpha g>1$,
\[
 \frac{7P}{10}-3
 =\frac{\alpha A_c}{10}+3(\alpha g-1)
 >\frac{\alpha A_c}{10}.
\]
Combining this identity with \eqref{eq:G-lower} yields
$G(P)>\alpha A_c/10$.  Proposition~\ref{prop:certified-hard-mode} then gives
\begin{align*}
 \Delta G(P)
 &>\frac{\alpha^2}{10}\sqrt r\,|\delta|A_c\\
 &>2\alpha^2tg
 =2(\alpha g)(\alpha t)>2\alpha t=2(P-C).
\end{align*}
The last strict step uses both $t>0$ and $\alpha g>1$; no sign is inherited
from the enlarged certificate box.  Hence
\[
 \Phi(1)=2(P-C)-\Delta G(P)<0,
\]
and therefore $\Phi(s)<0$.  The identity for $P(s+z)$ shows that
\[
 \Phi(s)=P\{2(s+z)-\Delta\},
\]
so $\Delta>2(s+z)$.  Finally,
\[
 4>W\ge\left(\frac{\Delta}{s+z}\right)^2>4,
\]
a contradiction.  Every one of the eight modes is excluded.
\end{proof}

\section{Four-layer bound and attainment}

The preceding propositions exhaust every possible block of five occupied
levels.  Proposition~\ref{prop:central-block-exclusion} excludes the central
block.  Layer reversal identifies the two side blocks, and
Proposition~\ref{prop:side-eight-modes} and
Proposition~\ref{prop:side-block-exclusion} exclude the normalized side block.  This
gives the uniform upper bound.  The rational construction below supplies the
matching lower bound.

\begin{theorem}[Uniform integer-occupied layer bound]
\label{thm:main}
For every instance satisfying the assumptions in
Section~\ref{sec:contract},
\[
|J_Z(I)|\le4.
\]
\end{theorem}

\begin{proof}
Suppose instead that $|J_Z(I)|\ge5$.
Proposition~\ref{prop:five-level-trichotomy} forces $J_Z(I)$ to be one of
the three five-level blocks in \eqref{eq:three-configs}.
Proposition~\ref{prop:central-block-exclusion} rules out the central block
$\{-2,-1,0,1,2\}$.  Layer reversal exchanges the two side blocks, so it
suffices to consider $\{-1,0,1,2,3\}$.
Proposition~\ref{prop:side-eight-modes} places every hypothetical occupation
of that block in one of eight explicit row modes, and
Proposition~\ref{prop:side-block-exclusion} excludes every mode.  All three
possibilities lead to a contradiction, proving $|J_Z(I)|\le4$.
\end{proof}

\begin{proposition}[Rational four-layer attainment]
\label{prop:rational-attainment}
There is an admissible rational instance $I_\sharp$ with
\[
 J_Z(I_\sharp)=\{-2,-1,0,1\}.
\]
\end{proposition}

\begin{proof}
Use the data in Appendix~\ref{app:certificate}.  The leading principal minor
and determinant of the quadratic matrix, and the determinant of the affine
lattice matrix, are
\[
 M_{11}=\frac{633}{25600}>0,
 \qquad
 \det M=\frac{9559}{327680000}>0,
 \qquad
 \det K=1.
\]
Thus $M\succ0$ and the affine lattice has full rank.  Exact completion of the
square and integer clamping give
\[
 R=1,\qquad A=0,\qquad B=8.
\]
The pulled-back inner metric and deterministic reduced basis are
\[
 H_z=\begin{pmatrix}11/4&11/8\\11/8&33/16\end{pmatrix},
 \quad D=16,
 \quad U=\begin{pmatrix}1&-1\\0&1\end{pmatrix},
\]
and
\[
 U^T(16H_z)U=\begin{pmatrix}44&-22\\-22&33\end{pmatrix}.
\]
The reduced Gram data therefore satisfy the frozen $\delta=3/4$ conditions
and give
\[
 \alpha=\frac{11}{4},\qquad
 \mu=-\frac12,
 \qquad
 \beta=\frac{11}{8}.
\]
Moreover, $\xi=(-3/4,-1/2)$; the tie-to-larger rule gives
$N_1=N_2=0$ and $e_1=e_2=-1/2$.  Hence
\[
 E_B=\alpha e_1^2+\beta e_2^2=\frac{33}{32}>1,
\]
so this instance lies in the strict-failure branch of the theorem.

In reduced coordinates the affine map is
\[
 \mathbf q=(8+6t-j,t).
\]
For a fixed safe level $j\in\{-3,\ldots,3\}$, the closed slab is equivalent
to the finite integer interval
\[
 \left\lceil\frac{j-8}{6}\right\rceil
 \le t\le
 \left\lfloor\frac j6\right\rfloor.
\]
Substitution into the original quadratic gives the exact row minima
\[
\begin{array}{c|rrrrrrr}
j&-3&-2&-1&0&1&2&3\\ \hline
204800\min F&15919&-553&-6897&-19889&-11129&7759&36775.
\end{array}
\]
The negative entries occur precisely on levels $-2,-1,0,1$; closed
feasibility includes their attaining integer points.  The positive entries
exclude every other safe level.  Thus
$J_Z(I_\sharp)=\{-2,-1,0,1\}$.
\end{proof}

\begin{corollary}[Exact extremum]
\label{cor:sharp-constant}
For the class of admissible instances,
\[
 \boxed{L_Z^*=4}.
\]
\end{corollary}

\begin{proof}
Theorem~\ref{thm:main} gives $L_Z^*\le4$, whereas
Proposition~\ref{prop:rational-attainment} gives
\[
 4=|J_Z(I_\sharp)|
 \le\sup_I|J_Z(I)|=L_Z^*.
\]
The two inequalities force $L_Z^*=4$.
\end{proof}

\begin{remark}[Three different constants]
The result does not alter the five-level sharp bound for real layer
intersection or the seven-level symmetric safe candidate list.  In the
attainment fixture the three counts are, respectively, four occupied integer
levels, five real-intersection levels, and seven safe candidates.
\end{remark}

\section{Reproducibility of the exact certificate}

The computer-assisted part consists of a finite change of basis from power
coefficients to Bernstein coefficients.  The ancillary program uses exact
rational arithmetic and evaluates all coefficients rather than samples of the
parameter box.  For each of the eight modes it produces degree $(4,6,8,6)$
and $2{,}205$ coefficients.  All coefficients are positive, and the smallest
is $625/2048$.  Appendix~\ref{app:certificate} records the conversion
formula, parameter box, mode table, exact minima, and attainment data needed
to reproduce the calculation.  The proof relies on the rational coefficient
inequalities, not on a digest or an unpublished software result.

\section{Relation to prior work and limitations}

Directional arithmetic width already provides the general counting object
\cite{DeLoeraMarstersONeill2025}.  Theorem~\ref{thm:main} concerns one direction
selected by a reduction algorithm and does not introduce that definition.
The usual planar flatness theorems do not apply directly: the body need not be
hollow, and the selected functional need not minimize lattice width.  Babai's
nearest-plane method and Lenstra-type hyperplane recursion supply the broader
algorithmic setting \cite{Lenstra1983,Babai1986,AgrellEtAl2002}.  The sharp
constant here comes from the interaction between the selected direction, the
shared disk--slab cap, and strict failure of the chosen Babai point.

The theorem assumes a rational planar
positive-definite quadratic body, a full-rank affine lattice coset, a closed
slab, and the stated deterministic reduction.  The theorem does not address
higher dimension, floating-point implementations, or arbitrary convex bodies.

\section*{Author contributions}

Xinyue Liu conceived the problem, developed the proof, implemented the exact
certificate computation, verified the resulting identities, and wrote the
manuscript.

\section*{Competing interests}

The author declares no competing interests.

\appendix
\input{appendix_exact_certificate}

\bibliographystyle{plain}
\bibliography{references}

\end{document}

%% file: appendix_exact_certificate.tex
\section{Exact Bernstein certificate and extremal example}
\label{app:certificate}

The side-case proof uses one finite symbolic calculation.  The calculation
converts a polynomial with rational coefficients to the Bernstein basis; it
does not use floating-point optimization or sampling.

\subsection{Mode data}

For a side configuration normalized to rows $\{-1,0,1,2,3\}$, write
$e=|e_1|$, $h=|\mu|$, $\vartheta=e_2$, and $r=\beta/\alpha$.  The proof gives
\[
 e>\frac{47}{100},\qquad h>\frac{11}{25},\qquad
 \vartheta>\frac{23}{50},
\]
and
\[
 \frac{64}{25}<\alpha\le\frac{72}{25},\qquad
 \frac43<\beta\le\frac{36}{25},\qquad
 \frac12\le r<\frac9{16}.
\]
The two sign families and their complete integer cells are
\[
\begin{array}{c|c|c}
 \text{family} & t_3 & t_2\\ \hline
 \mu=+h & -1 & -2,-1,0,1\\
 \mu=-h &  2 & 0,1,2,3.
\end{array}
\]
For each of the eight modes set
\[
\begin{aligned}
a_3&=t_3-e+3\mu,& b_3&=3-\vartheta,\\
a_2&=t_2-e+2\mu,& b_2&=2-\vartheta,\\
d&=t_3-t_2+\mu,\\
p&=a_3^2+r b_3^2,&
t&=a_3d+r b_3,\\
g&=e^2+r\vartheta^2,&
\delta&=a_3b_2-a_2b_3.
\end{aligned}
\]
The actual-domain sign checks are
\[
t>0,\qquad 7p-30g>0,\qquad r>0,\qquad g>0.
\]

\subsection{Certified polynomial}

The hard modes reduce to
\[
 \sqrt r\,|\delta|(7p-30g)>20tg.
\]
Because all factors have the signs displayed above, it is enough to prove
strict positivity of
\[
 \mathcal B=r\delta^2(7p-30g)^2-400t^2g^2.
\]
Introduce nonnegative box coordinates by
\[
\begin{aligned}
r&=\frac34-h^2+R,&0&\le R\le\frac1{16},\\
e&=\frac12-E,&0&\le E\le\frac3{100},\\
h&=\frac12-H,&0&\le H\le\frac3{50},\\
\vartheta&=\frac12-Q,&0&\le Q\le\frac1{25}.
\end{aligned}
\]
This closed rational box contains the true parameter domain.  With
\[
(R,E,H,Q)=\left(\frac{x_0}{16},\frac{3x_1}{100},
\frac{3x_2}{50},\frac{x_3}{25}\right),
\]
the power polynomial has tensor degree $d=(4,6,8,6)$.  If $a_\nu$ denotes
its power coefficient, the degree-$d$ Bernstein coefficient is
\[
b_k=\sum_{\nu\le k}a_\nu
\prod_{i=0}^3\frac{\binom{k_i}{\nu_i}}{\binom{d_i}{\nu_i}}.
\]
Every factor in this identity is an integer or a rational number.  The
Bernstein basis is nonnegative and forms a partition of unity on the unit
cube, so $\mathcal B\ge\min_k b_k$.

The exact results are
\[
\begin{array}{c|r|r|r}
\text{family}&t_2&\#\{b_k\}&\min b_k\\ \hline
\mu=+h&-2&2205&1445625/2048\\
\mu=+h&-1&2205&625/2048\\
\mu=+h& 0&2205&625/2048\\
\mu=+h& 1&2205&1445625/2048\\
\mu=-h& 0&2205&1445625/2048\\
\mu=-h& 1&2205&625/2048\\
\mu=-h& 2&2205&625/2048\\
\mu=-h& 3&2205&1445625/2048.
\end{array}
\]
Thus all $8\cdot2205=17640$ coefficients are strictly positive and
\[
\mathcal B\ge\frac{625}{2048}>0.
\]
The ancillary program \texttt{exact\_bernstein\_certificate.py} evaluates
every coefficient in this table using Python's exact rational arithmetic.  Its
canonical result digest is
\begin{center}
\small\texttt{654683d789a12c802315c5ecbb33b1cc4ef581175875e18f329b89e0ce3656dd}.
\end{center}
This digest covers the ordered mode records, coefficient counts, minima, and
per-mode coefficient digests; it is not used as a substitute for the rational
coefficient checks themselves.

\subsection{Four-layer attainment fixture}

For the attaining instance, define the original feasible set by
\begin{equation}
 F(\mathbf q)=\mathbf q^TM\mathbf q+2r_0^T\mathbf q+c_0\le0,
 \qquad \mathbf q=q_0+Kz,\qquad z\in\mathbb Z^2,
 \label{eq:fixture-original}
\end{equation}
with the data
\[
q_0=(8,0),\qquad
K=\begin{pmatrix}6&5\\1&1\end{pmatrix},\qquad
M=\begin{pmatrix}
633/25600&-869/6400\\
-869/6400&9559/12800
\end{pmatrix},
\]
\[
r_0=\begin{pmatrix}-1059/5120\\11297/10240\end{pmatrix},
\qquad c_0=13351/8192,
\]
and the closed slab $0\le q_1\le8$.  Exact arithmetic gives
\[
\det K=1,\qquad M\succ0,\qquad
q_c=-M^{-1}r_0=(160,1215/44),\qquad
r_0^TM^{-1}r_0-c_0=1.
\]
The rational inner-ellipse centre in the original plane and its pullback are
\[
 q_E=(4,-3/4),\qquad z_E=K^{-1}(q_E-q_0)=(-1/4,-1/2).
\]
The pulled-back metric, its canonical denominator, and the deterministic
reduced basis are
\[
H_z=\begin{pmatrix}11/4&11/8\\11/8&33/16\end{pmatrix},\qquad D=16,
\]
\[
U=\begin{pmatrix}1&-1\\0&1\end{pmatrix},
\qquad
U^T(16H_z)U=\begin{pmatrix}44&-22\\-22&33\end{pmatrix}.
\]
Thus
\[
 \alpha=\frac{11}{4},\qquad \mu=-\frac12,\qquad
 \beta=\frac{11}{8},\qquad
 \xi=U^{-1}z_E=(-3/4,-1/2).
\]
The tie-to-larger rule gives $N_1=N_2=0$ and
$e_1=e_2=-1/2$, so the deterministic Babai energy is
\[
 E_B=\alpha e_1^2+\beta e_2^2=\frac{33}{32}>1.
\]

Because the Babai indices vanish, $n=(t,j)^T=U^{-1}z$.  Moreover
\[
 KU=\begin{pmatrix}6&-1\\1&0\end{pmatrix},\qquad
 \mathbf q=q_0+KU\binom tj=(8+6t-j,t).
\]
Five explicit feasible points occupy four distinct levels:
\[
\begin{array}{c|r|c|c|c}
j&t&z=U(t,j)^T&\mathbf q&F(\mathbf q)\\ \hline
-2&-1&(1,-2)&(4,-1)&-553/204800\\
-1&-1&(0,-1)&(3,-1)&-6897/204800\\
 0&-1&(-1,0)&(2,-1)&-3113/204800\\
 0& 0&(0,0)&(8,0)&-19889/204800\\
 1& 0&(-1,1)&(7,0)&-11129/204800.
\end{array}
\]
For a fixed safe level $j\in\{-3,\ldots,3\}$, the closed slab is exactly
\[
 \left\lceil\frac{j-8}{6}\right\rceil
 \le t\le
 \left\lfloor\frac j6\right\rfloor.
\]
Evaluating \eqref{eq:fixture-original} on these finite intervals gives
\[
\begin{array}{c|rrrrrrr}
j&-3&-2&-1&0&1&2&3\\ \hline
204800\min F&15919&-553&-6897&-19889&-11129&7759&36775.
\end{array}
\]
Consequently the complete occupied set is
\[
J_Z=\{-2,-1,0,1\}.
\]
For comparison, the real nonintegral reduced point $n=(1/3,2)$ maps to
$z=(-5/3,2)$ and $\mathbf q=(8,1/3)$, where
\[
 F(\mathbf q)=-\frac{5201}{1843200}<0.
\]
Its normalized full-ellipse value is
\[
 \frac{1837999}{1843200}<1.
\]
This last fact separates the integer occupation count from real layer
intersection: the same fixture has five real-intersection layers but only
four integer-occupied layers.

%% file: references.bib
@article{DeLoeraMarstersONeill2025,
  author = {De Loera, Jes{\'u}s A. and Marsters, Brittney and O'Neill, Christopher},
  title = {An arithmetic measure of width for convex bodies},
  journal = {arXiv preprint arXiv:2509.04726},
  year = {2025},
  url = {https://arxiv.org/abs/2509.04726}
}

@article{DezaOnn1995,
  author = {Deza, Michel and Onn, Shmuel},
  title = {Lattice-free polytopes and their diameter},
  journal = {Discrete \& Computational Geometry},
  volume = {13},
  number = {1},
  pages = {59--75},
  year = {1995},
  doi = {10.1007/BF02574028}
}

@article{Hurkens1990,
  author = {Hurkens, C. A. J.},
  title = {Blowing up convex sets in the plane},
  journal = {Linear Algebra and its Applications},
  volume = {134},
  pages = {121--128},
  year = {1990},
  doi = {10.1016/0024-3795(90)90010-A}
}

@misc{AverkovCodenottiFreyerHuang2026,
  author = {Averkov, Gennadiy and Codenotti, Giulia and Freyer, Ansgar and Huang, Kyle},
  title = {Exact Flatness Constant for One-Point Convex Bodies and the Discrete Isominwidth Problem: The Planar Case},
  year = {2026},
  eprint = {2604.27260},
  archivePrefix = {arXiv},
  primaryClass = {math.MG},
  note = {arXiv:2604.27260v2},
  url = {https://arxiv.org/abs/2604.27260}
}

@misc{CodenottiHallHofscheier2021,
  author = {Codenotti, Gabriele and Hall, Thomas and Hofscheier, Johannes},
  title = {Generalised Flatness Constants: A Framework Applied in Dimension 2},
  year = {2021},
  eprint = {2110.02770},
  archivePrefix = {arXiv},
  primaryClass = {math.MG},
  note = {arXiv:2110.02770},
  url = {https://arxiv.org/abs/2110.02770}
}

@article{Babai1986,
  author = {Babai, L{\'a}szl{\'o}},
  title = {On {Lov{\'a}sz}' lattice reduction and the nearest lattice point problem},
  journal = {Combinatorica},
  volume = {6},
  number = {1},
  pages = {1--13},
  year = {1986},
  doi = {10.1007/BF02579403}
}

@article{Lenstra1983,
  author = {Lenstra, Jr., Hendrik W.},
  title = {Integer programming with a fixed number of variables},
  journal = {Mathematics of Operations Research},
  volume = {8},
  number = {4},
  pages = {538--548},
  year = {1983},
  doi = {10.1287/moor.8.4.538}
}

@article{FinckePohst1985,
  author = {Fincke, Ulrich and Pohst, Michael},
  title = {Improved methods for calculating vectors of short length in a lattice, including a complexity analysis},
  journal = {Mathematics of Computation},
  volume = {44},
  pages = {463--471},
  year = {1985},
  doi = {10.1090/S0025-5718-1985-0777278-8}
}

@article{SchnorrEuchner1994,
  author = {Schnorr, Claus-Peter and Euchner, Martin},
  title = {Lattice basis reduction: Improved practical algorithms and solving subset sum problems},
  journal = {Mathematical Programming},
  volume = {66},
  pages = {181--199},
  year = {1994},
  doi = {10.1007/BF01581144}
}

@article{AgrellEtAl2002,
  author = {Agrell, Erik and Eriksson, Thomas and Vardy, Alexander and Zeger, Kenneth},
  title = {Closest point search in lattices},
  journal = {IEEE Transactions on Information Theory},
  volume = {48},
  number = {8},
  pages = {2201--2214},
  year = {2002},
  doi = {10.1109/TIT.2002.800499}
}

@article{KannanLovasz1988,
  author = {Kannan, Ravi and Lov{\'a}sz, L{\'a}szl{\'o}},
  title = {Covering minima and lattice-point-free convex bodies},
  journal = {Annals of Mathematics},
  volume = {128},
  number = {3},
  pages = {577--602},
  year = {1988},
  doi = {10.2307/1971436}
}
